\documentclass[12pt, reqno]{amsart}
\usepackage{amssymb,amsthm,amsfonts,amsmath, amscd}

\usepackage{hyperref}
\usepackage{mathrsfs}

\newcommand\Z{\mathbb Z}
\newcommand\C{\mathbb C}
\newcommand\R{\mathbb R}

\newcommand\GT{{\mathbb{GT}}}

\newcommand\La{\Lambda}

\newcommand\wt{\widetilde}
\newcommand\wh{\widehat}

\renewcommand\H{\mathscr H}
\newcommand\X{\mathscr X}

\newtheorem{theorem}{Theorem}[section]
\newtheorem{proposition}[theorem] {Proposition}

\newtheorem{corollary}[theorem]{Corollary}
\newtheorem{lemma}[theorem]{ Lemma}

\theoremstyle{definition}

\newtheorem{remark}[theorem]{Remark}

\begin{document}

\title[]{Projections of orbital measures, Gelfand--Tsetlin polytopes, and splines}

\author{Grigori Olshanski}
\address{Institute for Information Transmission Problems, Moscow, Russia;
\newline\indent Independent University of Moscow, Russia;
\newline\indent National Research University Higher School of Economics, Moscow, Russia}

\email{olsh2007@gmail.com}

\date{}

\maketitle

\begin{abstract}
The unitary group $U(N)$ acts by conjugations on the space $\mathscr H(N)$ of
$N\times N$ Hermitian matrices, and every orbit of this action carries a unique
invariant probability measure called an orbital measure. Consider the
projection of the space $\mathscr H(N)$ onto the real line assigning to an
Hermitian matrix its $(1,1)$-entry. Under this projection, the density of the
pushforward of a generic orbital measure is a spline function with $N$ knots.
This fact was pointed out by Andrei Okounkov in 1996, and the goal of the paper
is to propose a multidimensional generalization. Namely, it turns out that if
instead of the $(1,1)$-entry we cut out the upper left matrix corner of
arbitrary size $K\times K$, where $K=2,\dots,N-1$, then the pushforward of a
generic orbital measure is still computable: its density is given by a $K\times
K$ determinant composed from one-dimensional splines. The result can also be
reformulated in terms of projections of the Gelfand--Tsetlin polytopes.
\end{abstract}

\maketitle

\tableofcontents

\section{Introduction}\label{sect1}

\noindent{\bf Orbital measures.} Let $\H(N)$ be the space of $N\times N$
Hermitian matrices. For $K=1,\dots,N-1$, we denote by $p^N_K:\H(N)\to\H(K)$ the
linear projection consisting in deleting from the matrix $H\in\H(N)$ its last
$N-K$ rows and columns. We call $p^N_K(H)$, the image of $H$ under this
projection, the $K\times K$ {\it corner\/} of $H$.

The unitary group $U(N)$ acts on $\H(N)$ by conjugations, and because $U(N)$ is
compact, each orbit of this action carries a unique invariant probability
measure, which we call the {\it orbital measure\/}.  Given an orbital measure
$\mu$ on $\H(N)$, denote by $p^N_K(\mu)$ its pushforward under projection
$p^N_K$. Our goal is to describe $p^N_K(\mu)$.

The orbits in $\H(N)$ (and hence the orbital measures) can be indexed by
$N$-tuples of weakly increasing real numbers $X=(x_1\le\dots\le x_N)$, the
matrix eigenvalues. Let $\X(N)\subset\R^N$ denote the set of all such $X$'s.
Given $X\in\X(N)$, we write $O_X$ and $\mu_X$ for the corresponding orbit and
orbital measure, respectively.

Since $p^N_K(\mu_X)$ is a $U(K)$-invariant probability measure on $\H(K)$, it
can be uniquely decomposed into a continual convex combination of orbital
measures, governed by a probability measure $\nu_{X,K}$ on the parameter space
$\X(K)$. That is, $\nu_{X,K}$ is characterized by the property that, for an
arbitrary Borel subset $S\subseteq\X(N)$,
$$
(p^N_K(\mu_X))(S)=\int_{Y\in\X(K)}\mu_Y(S)\nu_{X,K}(dY).
$$
The measure $\nu_{X,K}$ can be called the {\it radial part\/} of measure
$p^N_K(\mu_X)$.

\medskip

\noindent{\bf Main result.} Denote by $\X^0(N)$ the interior of $\X(N)$; that
is, $\X^0(N)$ consists of $N$-tuples  of {\it strictly\/} increasing real
numbers. If $X\in\X^0(N)$, then $\nu_{X,K}$ is absolutely continuous with
respect to Lebesgue measure on $\X(K)\subset\R^K$, and the main result, Theorem
\ref{thm3.A}, gives an explicit formula for the density of $\nu_{X,K}$.

In the case $K=1$ the target space of the projection is the real line, and the
density in question coincides with a {\it B-spline\/}, a certain piecewise
polynomial function on $\R$ (this fact was observed by Andrei Okounkov). In the
general case, it turns out that the density of $\nu_{X,K}$ is expressed through
a $K\times K$ determinant composed from some B-splines.

As the reader will see, the proof of Theorem \ref{thm3.A} is straightforward
and elementary. The main reason why I believe the result may be of interest is
the very appearance of splines, which are objects of classical and numerical
analysis, in a problem of representation-theoretic origin.

\medskip

\noindent{\bf Gelfand--Tsetlin polytopes.} Before explaining a connection with
representation theory I want to give a different interpretation of the measure
$\nu_{X,K}$.

For $X\in\X(N)$ and $Y\in\X(N-1)$, write $Y\prec X$ or $X\succ Y$ if the
coordinates of $X$ and $Y$ {\it interlace\/}, that is
$$
x_1\le y_1\le x_2\le\dots \le x_{N-1}\le y_{N-1}\le x_N.
$$
Given $X\in\X(N)$, the corresponding {\it Gelfand--Tsetlin polytope\/} $P_X$ is
the compact convex subset in the vector space
$$
\R^{N-1}\times\R^{N-2}\times\dots\times\R=\R^{N(N-1)/2},
$$
formed by triangular arrays subject to the interlacement constraints:
$$
P_X:=\{(Y^{(N-1)},\dots,Y^{(1)})\in R^{N(N-1)/2}: X\succ
Y^{(N-1)}\succ\dots\succ Y^{(1)}\}.
$$
Consider the map assigning to a matrix $H\in O_X$ the array formed by the
collections of eigenvalues of its corners $p^N_{N-1}(H), p^N_{N-2}(H),\dots,
p^N_1(H)$. It is well known (see Corollary \ref{cor3.A} below) that this map
projects the orbit $O_X$ onto the polytope $P_X$ and takes $\mu_X$ to the
uniform measure on $P_X$ (that is, the normalized Lebesgue measure). Next,
given $K=1,\dots,N-1$, consider the natural projection $P_X\to\X(K)$ extracting
from the array $(Y^{(N-1)},\dots,Y^{(1)})$ its $K$th component $Y^{(K)}$. The
measure $\nu_{X,K}$ is nothing else than the pushforward of the uniform measure
under the latter projection.

\medskip

\noindent {\bf Discrete version of the problem: relative dimension in
Gelfand--Tsetlin graph.} Let $\GT_N:=\X(N)\cap\Z^N$ be the set of weakly
increasing $N$-tuples of integers. The elements of $\GT_N$ are in bijection
with the irreducible representations of the group $U(N)$: with
$X=(x_1,\dots,x_N)\in\GT_N$ we associate the irreducible representation $T_X$
with signature (=highest weight) $\wh X:=(x_N,\dots,x_1)$. Here we pass from
$X$ to $\wt X$, because the coordinates of signatures are usually written in
the descending order, see Weyl \cite{Weyl}.

Let $X\in\GT_N$ and consider the finite set $P_X^\Z:=P_X\cap Z^{N(N-1)/2}$
consisting of integral points in the polytope $P_X$. Let us replace the uniform
measure on $P_X$ by the uniform measure on $P^\Z_X$ (that is, the normalized
counting measure). Next, given $K=1,\dots,N-1$, we consider again the same
projection $P_X\to \X(K)$ as before and denote by $\nu^\Z_{X,K}$ the
pushforward of the uniform measure on $P^\Z_X$. Evidently, $\nu^\Z_{X,K}$ is a
probability measure with finite support.

Elements of $P^\Z_X$ are the {\it Gelfand--Tsetlin schemes\/} (also called
Gelfand--Tsetlin patterns) with top row $X$; they parameterize the elements of
Gelfand--Tsetlin basis in $T_X$. By the very definition of $\nu^\Z_{X,K}$, for
$Y\in\GT_K$, the quantity $\nu^\Z_{X,K}(Y)$ (the mass assigning by
$\nu^\Z_{X,K}$ to $Y$) equals the fraction of the schemes with the $K$th row
equal to $Y$. This quantity is the same as the relative dimension of the
isotypic component of $T_Y$ in the restriction of $T_X$ to the subgroup
$U(K)\subset U(N)$.

The {\it Gelfand--Tsetlin graph\/} has the vertex set
$\GT_1\sqcup\GT_2\sqcup\dots$ and the edges formed by couples $Y\prec X$. In
the terminology of Borodin--Olshanski \cite{BO-Adv-2012}, $\nu^\Z_{X,K}(Y)$ is
the {\it relative dimension\/} of the vertex $Y\in\GT_K$ with respect to the
vertex $X\in\GT_N$. In \cite{BO-Adv-2012}, we derived a determinantal formula
for the relative dimension (see also Petrov \cite{Petrov} for a different
proof). That formula can be viewed as a discrete version of the formula of
Theorem \ref{thm3.A}.

I first guessed the formula of Theorem \ref{thm3.A} by degenerating the
``discrete'' formula of \cite{BO-Adv-2012}. However, this is not an optimal way
of derivation, because the discrete case is much more difficult than the
continuous one. I am grateful to Alexei Borodin for the suggestion to study the
degeneration of the ``discrete'' formula. Note that from the comparison of the
measures $\nu_{N,K}$ and $\nu^\Z_{X,K}$ it is seen that the former should be
related to the latter by a scaling limit transition.

\section{Preliminaries}\label{sect2}

The {\it fundamental spline\/} with $n\ge2$ knots $y_1<\dots<y_n$ can be
characterized as the only function $a\mapsto M(a;y_1\dots,y_n)$ on $\R$ of
class $C^{n-3}$, vanishing outside the interval $(y_1,y_n)$, equal to a
polynomial of degree $\le n-2$ on each interval $(y_i,y_{i+1})$, and normalized
by the condition
\begin{equation*}
\int_{-\infty}^{+\infty} M(a; y_1,\dots,y_n)da=1.
\end{equation*}
Here is an explicit expression:
\begin{equation}\label{eq2.C}
M(a;y_1,\dots,y_n):=(n-1)\sum_{i:\,
y_i>a}\frac{(y_i-a)^{n-2}}{\prod\limits_{r:\, r\ne i}(y_i-y_r)}.
\end{equation}
In particular, for $n=2$
\begin{equation*}
M(a;y_1,y_2)=\frac{\mathbf1_{y_1\le a\le y_2}}{y_2-y_1}.
\end{equation*}

\begin{remark}

The above definition is taken from Curry--Schoenberg \cite{CS}. In the
subsequent publications, Schoenberg changed the term to {\it B-spline\/}. The
latter term became commonly used. However, in the modern literature, it more
often refers to the function
\begin{equation}\label{eq2.H}
B(a;y_1,\dots,y_n):=(y_n-y_1)\sum_{i:\,
y_i>a}\frac{(y_i-a)^{n-2}}{\prod\limits_{r:\, r\ne i}(y_i-y_r)}\,,
\end{equation}
which differs from $M(a;y_1,\dots,y_n)$ by the numerical factor
$(y_n-y_1)/(n-1)$; see, e.g., de Boor \cite{DeBoor} or Phillips
\cite{Phillips}. The normalization in \eqref{eq2.H} has its own advantages, but
we will not use it. Note also that $M(a;y_1,\dots,y_n)$ is a special case of
{\it Peano kernel\/}, see Davis \cite{Davis}, Faraut \cite{Faraut}.

\end{remark}

We need two well-known formulas relating $M(a;y_1,\dots,y_n)$ to divided
differences (see, e.g., \cite{CS}, \cite{Faraut}).

Recall that the {\it divided difference\/} of a function $f(x)$ on points
$y_1,\dots,y_n$ is defined recursively by
$$
f[y_1,y_2]=\frac{f(y_2)-f(y_1)}{y_2-y_1}\,, \quad
f[y_1,y_2,y_3]=\frac{f[y_2,y_3)-f[y_1,y_2]}{y_3-y_1}\,,
$$
and so on; the final step is
\begin{equation}\label{eq2.A}
f[y_1,\dots,y_n]=\frac{f[y_2,\dots,y_n]-f[y_1,\dots,y_{n-1}]}{y_n-y_1}\,.
\end{equation}

Next, set
$$
x_+^s=\begin{cases} x^s, & x>0\\
0, & x\le0. \end{cases}
$$

In this notation, the first formula in question is
\begin{equation}\label{eq2.B}
M(a; y_1,\dots,y_n)=(n-1)f[y_1,\dots,y_n], \qquad \textrm{where
$f(x):=(x-a)^{n-2}_+$},
\end{equation}
and the second formula is
\begin{equation}\label{eq2.D}
f[y_1,\dots,y_n]=\frac1{(n-1)!}\int M(a;y_1,\dots,y_n)f^{(n-1)}(a)da.
\end{equation}

In \eqref{eq2.D}, $f$ is assumed being a function on $\R$ with piecewise
continuous derivative of order $n-1$. In particular, \eqref{eq2.D} is
applicable to $f(x)=(x-t)_+^{n-1}$, which is used in the lemma below.

To shorten the notation, let us abbreviate $Y:=(y_1<\dots<y_n)$.

\begin{lemma}\label{lemma2.A}
Fix $n=2,3,\dots$ and an $n$-tuple $Y=(y_1<\dots<y_n)\in\X^0(N)$. For an
arbitrary $b\in\R$ set
$$
f_b(x):=(x-b)^{n-1}_+, \qquad x\in\R.
$$
One has
\begin{gather}
\int_{-\infty}^{c}M(a;Y)da=1-f_{c}[Y], \quad c\in\R,\label{eq2.E}\\
\int_{b}^{c}M(a;Y)da=f_{b}[Y]-f_{c}[Y], \quad b<c,\label{eq2.F}\\
\int_{b}^{+\infty}M(a;Y)da=f_{b}[Y], \quad b\in\R.\label{eq2.G}
\end{gather}
\end{lemma}

\begin{proof}
To check \eqref{eq2.F}, we apply \eqref{eq2.D} to $f(x)=f_b(x)-f_c(x)$, which
is justified, see the comment just after \eqref{eq2.D}. Then in the left-hand
side of \eqref{eq2.D} we get $f_b[Y]-f_c[Y]$. Next, observe that the $(n-1)$th
derivative of $f_t(x)$ equals $(n-1)!\mathbf1_{x\ge b}$, so that
$$
f^{(n-1)}(a)=(n-1)!(\mathbf1_{a\ge b}-\mathbf1_{a\ge c})=(n-1)!\mathbf1_{b\le
a<c}.
$$
Therefore, in the right-hand side we get $\int_b^c M(a;Y)da$, which proves
\eqref{eq2.F}.

Now \eqref{eq2.G} follows from \eqref{eq2.F} by setting $c=+\infty$, and
\eqref{eq2.E} follows from \eqref{eq2.G}, because the total integral of the
$M(a;Y)$ equals 1.
\end{proof}

\section{Projections of orbital measures}\label{sect3}.

We keep to the notation of Sections \ref{sect1} and \ref{sect2}

Given $X\in\X(N)$, the pushforward of the orbital measure $\mu_X$ under the map
$$
O_X\ni H\mapsto \textrm{the spectrum of $p^N_{N-1}(H)$}
$$
can be viewed as a probability measure on $\X(N-1)$ depending on $X$ as a
parameter; let us denote it by $\La^N_{N-1}(X,\,\cdot\,)$ or
$\La^N_{N-1}(X,dY)$. We regard $\La(X,dY)$ as a Markov kernel.

By classical Rayleigh's theorem, the eigenvalues of a matrix $H\in\X(N)$ and
its corner $p^N_{N-1}(H)$ interlace. Therefore, the measure
$\La^N_{N-1}(X,\,\cdot\,)$ is concentrated on the subset
\begin{equation}\label{eq3.A}
\{Y\in\X(N-1):Y\prec X\}\subset R^{N-1}.
\end{equation}

\begin{proposition}\label{prop3.A}
Assume $X=(x_1,\dots,x_N)\in\X^0(N)$. Then the measure
$\La^N_{N-1}(X,\,\cdot\,)$ is absolutely continuous with respect to Lebesgue
measure on the set \eqref{eq3.A}, and the density of
$\La^N_{N-1}(X,\,\cdot\,)$, denoted by $\La^N_{N-1}(X,Y)$, is given by
\begin{equation}\label{eq3.B}
\La^N_{N-1}(X,Y)=(N-1)!\frac{V(Y)}{V(X)}\mathbf1_{Y\prec X},
\end{equation}
where we use the notation
$$
V(X)=\prod_{j>i}(x_j-x_i)
$$
and the symbol\/ $\mathbf1_{Y\prec X}$ equals\/ $1$ or\/ $0$ depending on
whether $Y\prec X$ or not.
\end{proposition}

\begin{proof}
To the best of my knowledge, a published proof first appeared in Baryshnikov
\cite[Proposition 4.2]{Bar-PTRF-01}. However, the argument given in
\cite{Bar-PTRF-01} was known earlier: it is hidden in the first computation of
the spherical functions of the groups $SL(N,\C)$ due to Gelfand and Naimark,
see \cite[\S9]{GelfandNaimark-MIAN1950}. Note also that a more general result
can be found in Neretin \cite{Ner-MathSb2003}.

Here is a different proof. Consider the Laplace transform of the orbital
measure $\mu_X$:
\begin{equation}\label{eq3.N}
\wh \mu_X(Z):=\int e^{\operatorname{Tr}(ZH)}\mu_X(dH),
\end{equation}
where $Z$ is a complex $N\times N$ matrix. The integral in the right-hand side
is often called the {\it Harish-Chandra--Itzykson--Zuber integral\/}. Its value
is given by a well-known formula (see, e.g., Olshanski--Vershik \cite[Corollary
5.2]{OV-AMS96}):
\begin{equation}\label{eq3.K}
\wh \mu_X(Z)=
c_N\,\frac{\det[e^{z_ix_j}]_{i,j=1}^N}{\prod\limits_{j>i}(z_j-z_i)(x_j-x_i)}\,,
\end{equation}
where $z_1,\dots,z_N$ are the eigenvalues of $Z$ and
$$
\qquad c_N=(N-1)!(N-2)!\dots0!
$$
(note that the right-hand side of \eqref{eq3.K} does not depend on the
enumeration of the eigenvalues of $Z$).

The claim of the proposition is equivalent to the following equality: Assume
that the entries in the last row and column of $Z$ equal 0, so that $Z$ has the
form
\begin{equation}\label{eq3.M}
Z=\begin{bmatrix} \wt Z & 0\\ 0& 0 \end{bmatrix},
\end{equation}
where $\wt Z$ is a complex matrix of size $(N-1)\times(N-1)$; then
\begin{equation}\label{eq3.L}
\wh\mu_X(Z)=\frac{(N-1)!}{V(X)}\int_{Y\prec X}V(Y)\wh\mu_Y(\wt Z)dY.
\end{equation}

To prove \eqref{eq3.L}, consider the matrix $T:=[e^{z_ix_j}]$ in the right-hand
side of \eqref{eq3.K}. Since $Z$ has the form \eqref{eq3.M}, at least one of
the eigenvalues $z_1,\dots,z_N$ equals 0. It is convenient to slightly change
the enumeration and denote the eigenvalues as $z_0=0,z_1,\dots,z_{N-1}$. In
accordance to this we will assume that the row number $i$ of $T$ ranges over
$\{0,\dots,N-1\}$ while the column index $j$ ranges over $\{1,\dots,N\}$. Since
$z_0=0$, the 0th row of $T$ is $(1,\dots,1)$. Let us subtract the $(N-1)$th
column from the $N$th one, then subtract the $(N-2)$th column from the
$(N-1)$th one, etc. This gives $ \det T=\det \wt T$, where $\wt T$ stands for
the matrix of order $N-1$ with the entries
$$
\wt T_{i,j}=e^{z_i x_{j+1}}-e^{z_ix_j}=z_i\int_{x_i}^{x_j}e^{z_iy_j}dy_j,
\qquad i,j=1,\dots,N-1.
$$
It follows
$$
\det\wt T=z_1\dots z_N\,\int_{Y\prec X}dY\det[e^{z_iy_j}]_{i,j=1}^{N-1},
$$
so that
$$
\wh\mu_X(Z)=c_N\,\frac{z_1\dots z_N\int_{Y\prec
X}dY\det[e^{z_iy_j}]_{i,j=1}^{N-1}}{\prod\limits_{N-1\ge j>i\ge0}(z_j-z_i)\cdot
V(X)}\,.
$$
Next, because $z_0=0$, the product over $j>i$ in the denominator equals
$$
z_1\dots z_N\prod_{N-1\ge j>i\ge1}(z_j-z_i),
$$
so that the product $z_1\dots z_N$ is cancelled out. Taking into account the
fact that $\wh\mu_Y(\wt Z)$ is given by the determinantal formula similar to
\eqref{eq3.K} and using the obvious relation $c_N=(N-1)!c_{N-1}$ we finally get
the desired equality \eqref{eq3.L}.

\end{proof}

{}From Proposition \ref{prop3.A} it is easy to deduce the following corollary
(see also \cite[Proposition 4.7]{Bar-PTRF-01} and Defosseux \cite{Defosseux}).

\begin{corollary}\label{cor3.A}
Fix $X\in\X^0(N)$ and let $H$ range over $O_X$. The map assigning to $H$ the
collection of the eigenvalues of the corners $p^N_K(H)$, where
$K=N-1,N-2,\dots,1$, projects $O_X$ onto the Gelfand--Tsetlin polytope $P_X$
and takes the measure $\mu_X$ to the Lebesgue measure multiplied by the
constant
$$
\frac{(N-1)!(N-2)!\dots0!}{V(X)}.
$$
In particular, the volume of $P_X$ in the natural coordinates is equal to the
inverse of the above quantity.
\end{corollary}

Recall that $\nu_{X,K}$ stands for the radial part of the $K\times K$ corner of
the random matrix $H\in O_X$, driven by the orbital measure $\mu_X$ (see
Section \ref{sect1}), and $M(a; y_1,\dots,y_n)$ denotes the fundamental spline
with $n$ knots $y_1,\dots,y_n$ (see \eqref{eq2.C} and \eqref{eq2.B}).

\begin{theorem}\label{thm3.A}
Fix $X=(x_1,\dots,x_N)\in\X^0(N)$. For any $K=1,\dots,N-1$, the measure
$\nu_{X,K}$ on $\X(K)$ is absolutely continuous with respect to Lebesgue
measure and has the density
\begin{equation}\label{eq3.C}
M(a_1,\dots,a_K;x_1,\dots,x_N):=c_{N,K}\,
\frac{V(A)\,\det\left[M(a_j;x_i,\dots,x_{N-K+i})\right]_{i,j=1}^K}
{\prod\limits_{(j,i):\,j-i\ge N-K+1}(x_j-x_i)}\,,
\end{equation}
where
$$
c_{N,K}=\prod_{i=1}^{K-1}\binom{N-K+i}{i}.
$$
\end{theorem}

Note that for $K=1$ the right-hand side reduces to the fundamental spline with
knots $x_1,\dots,x_N$. Thus, in the case $K=1$ the theorem says that the
density of the measure $\nu_{N,1}$ on $\R$ coincides with the spline
$M(a;x_1,\dots,x_N)$. This simple but important claim is due to Andrei
Okounkov, see \cite[Proposition 8.2]{OV-AMS96}.

\begin{proof}
We argue by induction on $K$, starting with $K=N-1$ and ending at $K=1$.

{\it Step\/} 1. Examine the case $K=N-1$, which is the base of induction. We
have $\nu_{X,N-1}(dA)=\La^N_{N-1}(X,dA)$. By proposition \ref{prop3.A}, the
measure $\La^N_{N-1}(X,\,\cdot\,)$ on $\X(N-1)$ is absolutely continuous with
respect to Lebesgue measure and has density $\La^N_{N-1}(X,A)$ given by
\eqref{eq3.B}. Thus, we have to check that $\La^N_{N-1}(X,A)$ coincides with
the quantity $M(a_1,\dots,a_{N-1}; x_1,\dots,x_N)$ given by the right-hand side
of \eqref{eq3.C}, where we have to take $K=N-1$. That is, the desired equality
has the form
$$
(N-1)!\frac{V(A)}{V(X)}\mathbf1_{A\prec X}=c_{N,N-1}\,
\frac{V(A)\,\det\left[M(a_j; x_i,x_{i+1})\right]_{i,j=1}^{N-1}}
{\prod\limits_{(j,i):\,j-i\ge 2}(x_j-x_i)}\,.
$$

Since $c_{N,N-1}=(N-1)!$, the desired equality reduces to
$$
\det\left[M(a_j; x_i,x_{i+1})\right]_{i,j=1}^{N-1} =\frac{\mathbf1_{A\prec
X}}{(x_2-x_1)(x_3-x_2)\dots(x_N-x_{N-1})}\,.
$$

Observe that the $(i,j)$-entry in the determinant is the quantity
$$
M(a_j;x_i,x_{i+1})=\frac{\mathbf1_{x_i\le a_j\le x_{i+1}}}{x_{i+1}-x_i},
$$
which vanishes unless $a_j\in[x_i, x_{i+1}]$. Since $a_1\le\dots\le a_{N-1}$,
the determinant vanishes unless $A\prec X$. Furthermore, if $A\prec X$, then
the matrix under the sign of determinant is diagonal, so the determinant equals
the product of the diagonal entries, which equals
$$
\frac1{(x_2-x_1)(x_3-x_2)\dots(x_N-x_{N-1})},
$$
as required.

{\it Step 2\/}. Given $K=1,\dots,N-1$, we consider the superposition of Markov
kernels
$$
\La^N_K:=\La^N_{N-1}\La^{N-1}_{N-2}\dots \La^{K+1}_K.
$$
In more detail, the result is a Markov kernel on $\X(N)\times\X(K)$ given by
$$
\La^N_K(X,dA)=\int_{}\La^N_{N-1}(X,dY^{(N-1)})\La^{N-1}_{N-2}(Y^{(N-1)},dY^{(N-2)})
\dots\La^{K+1}_K(Y^{(K+1)},dA),
$$
where the integral is taken over variables $Y^{(N-1)},\dots,Y^{(K+1)}$.
Obviously, $\La^N_K(X,dA)=\nu_{X,K}(dA)$, which entails the recurrence relation
\begin{equation}\label{eq3.D}
\nu_{X,K-1}=\nu_{X,K}\La^K_{K-1}, \qquad K=N-1,N-2,\dots,2,
\end{equation}
where, by definition, $\nu_{X,K}\La^K_{K-1}$ is the measure on $\X(K-1)$ given
by
\begin{equation}\label{eq3.E}
(\nu_{X,K}\La^K_{K-1})(dB)=\int_{A\in\X(K)} \nu_{X,K}(dA) \La^K_{K-1}(A,dB).
\end{equation}

{\it Step\/} 3. Assume now that the claim of the theorem holds for some $K\ge
2$ and deduce from this that it also holds for $K-1$. To do this we employ
\eqref{eq3.D} and \eqref{eq3.E}.

First of all, \eqref{eq3.D} and \eqref{eq3.E} imply  that $\nu_{X,K-1}$ is
absolutely continuous with respect to Lebesgue measure on $\X(K-1)$ and has the
density
\begin{equation}\label{eq3.F}
(\nu_{X,K}\La^K_{K-1})(B)=\int_{A\in\X(K)} \nu_{X,K}(dA) \La^K_{K-1}(A,B),
\qquad B\in\X(K-1).
\end{equation}

Let us compute the integral explicitly. By the induction assumption,
$\nu_{X,K}$ is absolutely continuous and has density \eqref{eq3.C}. Therefore,
integral \eqref{eq3.F} can be written in the form
$$
\int_{A\in\X^0(K)} M(a_1,\dots,a_K; x_1,\dots,x_N)\La^K_{K-1}(A,B)da_1\dots
da_K.
$$
Write $B=(b_1,\dots,b_K)$. Substituting the explicit expression for
$\La^K_{K-1}(A,B)$ given by Proposition \ref{prop3.A} we rewrite this as
\begin{equation}\label{eq3.G}
(K-1)!V(B)\int_A \frac{M(a_1,\dots,a_K; x_1,\dots,x_N)}{V(A)}da_1\dots da_K,
\end{equation}
where the integration domain is
\begin{equation}\label{eq3.H}
-\infty<a_1\le b_1, \quad \dots, \quad b_i\le a_{i+1}\le b_{i+1}, \quad
\dots,\quad  b_{k-1}\le a_K<+\infty.
\end{equation}

Next, plug in into \eqref{eq3.G} the explicit expression for $M(a_1,\dots,a_K;
x_1,\dots,x_N)$ given by \eqref{eq3.C}. Then the factor $V(A)$ is cancelled out
and we get
\begin{equation}\label{eq3.I}
\frac{c_{N,K}(K-1)!V(B)}{\prod\limits_{(j,i):\,j-i\ge N-K+1}(x_j-x_i)}\int_A
\det\left[M(a_j; x_i,\dots,x_{N-K+i})\right]_{i,j=1}^K da_1\dots da_K
\end{equation}
with the same integration domain \eqref{eq3.H}.

Put aside the pre-integral factor in \eqref{eq3.I} and examine the integral
itself. It can be written as a $K\times K$ determinant,
\begin{equation*}
\det[F(i,j)]_{i,j=1}^K,
\end{equation*}
where
$$
F(i,j):=\int_{b_{j-1}}^{b_j}M(a; Y_i)da
$$
and
$$
Y_i:=(x_i,\dots,x_{N-K+i})
$$
with the understanding that $b_0=-\infty$ and $b_K=+\infty$.

We are going to prove that
\begin{multline}\label{eq3.J}
\det[F(i,j)]_{i,j=1}^K=(N-K+1)^{K-1}\prod_{i=1}^{K-1}(x_{N-K+i+1}-x_i) \\
\times\det[M(b_j;x_i,\dots,x_{N-K+i+1})]_{i,j=1}^{K-1}.
\end{multline}
This will justify the induction step, because
$$
c_{N,K}=c_{N,K-1}\cdot\frac{(N-K+1)^{K-1}}{(K-1)!}
$$
and
$$
\frac{\prod_{i=1}^{K-1}(x_{N-K+i+1}-x_i)}{\prod\limits_{(j,i):\,j-i\ge
N-K+1}(x_j-x_i)}=\frac1{\prod\limits_{(j,i):\,j-i\ge N-K+2}(x_j-x_i)}\,.
$$

{\it Step\/} 4. It remains to prove \eqref{eq3.J}. We evaluate the quantities
$F(i,j)$ using Lemma \ref{lemma2.A}, where we substitute $n=N-K+1$ and $Y=Y_i$.
Then we get that the matrix entries $F(i,j)$ are given by the following
formulas:

$\bullet$ The entries of the first column have the form
$$
F(i,1)=1-f_{b_1}[Y_i] \qquad \textrm{by \eqref{eq2.E}}.
$$

$\bullet$ The entries of the $j$th column, $2\le j\le K-1$, have the form
$$
F(i,j)=f_{b_{j-1}}[Y_i]-f_{b_j}[Y_i] \qquad \textrm{by \eqref{eq2.F}}.
$$

$\bullet$ The entries of the last column have the form
$$
F(i,K)=f_{b_K}[Y_i] \qquad \textrm{by \eqref{eq2.G}}.
$$

We have $\det F=\det G$, where the $K\times K$ matrix $G$ is defined by
$$
G(i,j):=F(i,j)+\dots+F(i,K).
$$
The entries of the matrix $G$ are
$$
G(i,1)=1, \qquad G(i,j)=f_{b_{j-1}}[Y_i], \quad 2\le j\le K.
$$

Next, we get $\det G=\det H$ with the $(K-1)\times(K-1)$ matrix $H$ defined by
$$
H(i,j):=F(i+1,j+1)-F(i,j), \qquad 1\le i,j\le K-1.
$$
Observe now that
$$
H(i,j)=f_{b_j}[Y_{i+1}]-f_{b_j}[Y_i],
$$
which can be rewritten as
\begin{gather}
H(i,j)=(x_{N-K+i+1}-x_i)\frac{f_{b_j}[x_{i+1},\dots
x_{N-K+i+1}]-f_{b_j}[x_i,\dots,x_{N-K+i}]}{x_{N-K+i+1}-x_i}\\
=(x_{N-K+i+1}-x_i)f_{b_j}[x_i,\dots,x_{N-K+i+1}] \quad \textrm{by \eqref{eq2.A}}\\
=\frac1{N-K+1}\,(x_{N-K+i+1}-x_i)M(b_j;x_i,\dots,x_{N-K+i+1}) \quad \textrm{by
\eqref{eq2.B}}.
\end{gather}
This shows that the determinant $\det H=\det[H(i,j)]_{i,j=1}^{K-1}$ equals the
right-hand side of \eqref{eq3.J}. Since $\det H=\det F$, this completes the
proof.

\end{proof}

\section{Acknowledgement}

I am grateful to Jacques Faraut for valuable comments. The work was partially
supported by a grant from Simons Foundation (Simons--IUM Fellowship) and the
project SFB 701 of Bielefeld University.

\end{document}